\documentclass[12pt]{amsart}
\usepackage{amsfonts}
\usepackage{amssymb,latexsym}
\usepackage{enumerate}
\usepackage{mathrsfs}
\usepackage{amsmath}
\usepackage{amsthm}
\usepackage{hyperref}
\usepackage[square, comma, sort&compress, numbers]{natbib}
\usepackage{bm}
\usepackage{fancyhdr}
\usepackage{lastpage}
\usepackage{layout}
\usepackage{ulem}
\usepackage{extarrows}
\newtheorem{theorem}{Theorem}[section]

\newtheorem{lemma}[theorem]{Lemma}
\newtheorem{definition}[theorem]{Definition}
\numberwithin{equation}{section}

\DeclareMathOperator \End{End}

\def\qed{\vbox{\hrule
\hbox{\vrule\hbox to 5pt{\vbox to 8pt{\vfil}\hfil}\vrule}\hrule}}

\newcommand{\beg}{\begin{eqnarray*}}
\newcommand{\begn}{\begin{eqnarray}}
\newcommand{\en}{\end{eqnarray*}}
\newcommand{\enn}{\end{eqnarray}}

\newcommand{\tr}{\mbox{\rm tr\,}}

\begin{document}

%
%
%
%
%
%


\title{\textsc{Generalized Donaldson's functionals and related nonlinear partial differential equations }}
\author{Chuanjing Zhang and  Xi Zhang}
\address{Chuanjing Zhang\\School of Mathematical Sciences,\\
University of Science and Technology of China,\\
Hefei, 230026,P.R. China\\}\email{chjzhang@mail.ustc.edu.cn}
\address{Xi Zhang\\School of Mathematical Sciences,\\
University of Science and Technology of China,\\
Hefei, 230026,P.R. China\\ } \email{mathzx@ustc.edu.cn}

\thanks{The second author is the corresponding author. The authors are partially supported by NSF in China No.11625106,
11801535 and 11721101. The research was partially supported by the project "Analysis and Geometry on Bundle" of Ministry of Science and Technology of the People's Republic of China, No.SQ2020YFA070080.}
\maketitle

\vspace{-0.5cm}
\begin{abstract}
In this paper, we introduce a family of generalized Donaldson's functional on holomorphic vector bundles, whose Euler-Lagrange equations are a vector bundle version of the complex $k$-Hessian equations. We also discuss the uniqueness of solutions to these equations.
\end{abstract}

\medskip

\section{Introduction}

\medskip

Let $(M, \omega )$ be an $n$-dimensional compact K\"ahler manifold and $(E, \overline{\partial}_{E})$ be a holomorphic vector bundle over $M$. For any $B\in \Omega^{2}(\End(E))$, that is any $\End(E)$-valued 2-form, we denote
\begin{equation}
B^{k}:=\underbrace{B\wedge B\cdots\wedge B}_k,
\end{equation}
with $k\in \mathbb N$ and $k\leq n$. For any given Hermitian metric $H$ on $E$, we use $D_{H}$ to denote the Chern connection with respect to the metric $H$ and the holomorphic structure, and $ F_H$ is the curvature tensor of the Chern connection $D_H$.

\medskip

Let $H_{0}$ be a fixed Hermitian metric on $E$ and $H(s)$  a path connecting metrics $H_{0}$ and another Hermitian metric $H$. For any integer $0<k\leq n$, we introduce the following generalized Donaldson's functionals
 \begin{equation}
 \mathcal{M} _{E, k}^{0}(H_{0}, H)=\int_{0}^{1}\int_{M} \tr \left[\left(\sqrt{-1}F_{H(s)}\right)^{k}\, H(s)^{-1}\,\frac{\partial H(s)}{\partial s}\right]\wedge \frac{\omega^{n-k}}{(n-k)!}\,ds
 \end{equation}
 and
 \begin{equation}\label{Fk}
  \mathcal{M} _{E, k }(H_{0}, H)= \mathcal{M} _{E, k}^{0}(H_{0}, H)-\int_{M} \lambda_{E, k}\cdot \log \det (H_{0}^{-1}H)\, \frac{\omega^{n}}{n!},
 \end{equation}
where the constant $\lambda_{E, k}=(2\pi )^{k}\,(k!)\,\frac{ch_{k}(E)\cdot \left[\frac{\omega^{n-k}}{(n-k)!}\right]}{{\rm rank}(E)\cdot {\rm vol}(M, \omega )}$.

\medskip

As it will be shown in \S 2, the functional $\mathcal{M} _{E, k }(H_{0}, H)$ is well defined, i.e. it is independent of the choice of path $H(s)$. Indeed, when $k=1$, the above functional $\mathcal{M} _{E, 1}(H_{0}, H)$ is just the well-known {\it Donaldson's functional}, which was introduced by Donaldson (\cite{Donaldson85}, \cite{Ko1}) in his groundbreaking work on the study of the Hermitian-Yang-Mills connection. In \cite{Donaldson85}, it is an important observation that the Euler-Lagrange equation of the Donaldson's functional is just the Hermitian-Yang-Mills equation
\begin{equation}\label{HYM}
\sqrt{-1}\, \Lambda_{\omega } F_{H}=\lambda_{E, 1} {\rm Id}_{E}.
\end{equation}
The solvability of the above equation has many interesting geometric applications. Motivated by the study of the inequalities between Chern numbers, we introduce the generalized Donaldson's functional by increasing the power for the curvature tensor $F_H$, which corresponds to higher order Chern classes.

\medskip

In \S 2, we will show that the Euler-Lagrange equation of the generalized Donaldson's functional $\mathcal M_{E, k }$ is given by
 \begin{equation}\label{Bk}
\frac{\left(\sqrt{-1}F_{H}\right)^{k}\wedge \frac{\omega^{n-k}}{(n-k)!}}{\frac{\omega^{n}}{n!}}=\lambda_{E, k}\,{\rm Id}_{E}.
\end{equation}
It is easy to see that, when $k=1$, the above equation is just the standard Hermitian-Yang-Mills equation (\ref{HYM}). On the other hand, if $(E, \overline{\partial}_{E})$ is a line bundle, the equation (\ref{Bk}) reduces to the complex $k$-Hessian equation. Given this connection, we will simply call the above equation as the {\it vector bundle valued $k$-Hessian equation}. In fact, such type of equations were considered in various settings. When $k=n$, the equation (\ref{Bk}) is a vector bundle version of the complex Monge-Amp\`ere equation introduced by Pingali (\cite{Pin}) from the viewpoint of moment map, which generalizes the dHYM equations studied in \cite{CJY, CXY} to vector bundles. When $k=2$, equations as (\ref{Bk}) also arise naturally in the Hull-Strominger system from theoretical physics. One crucial equation in the system is the so-called anomaly cancellation equation, which can be viewed as a constraint on the second Chern class and it has similar form as (\ref{Bk}). The system has been studied extensively in recent years, see for example \cite{FY1, FY2, H, PPZ1, PPZ2, PPZ3, PPZ4, S}.

\medskip

For the Hermitian-Yang-Mills equation (\ref{HYM}), the classical Donaldson-Uhlenbeck-Yau theorem (\cite{NS65}, \cite{Donaldson85}, \cite{UY}) states that a holomorphic bundle admits the Hermitian-Einstein metric if and only if it is poly-stable in the sense of Mumford-Takemoto. There are  many  interesting and important works related (\cite{ag,Bi,Bu,Donaldson87,HIT,JZ, LN2, LZ, LZZ, LZZ2,LY87, Mo1, Mo2, Mo3,SIM,SIM2}, etc.). The complex Monge-Amp\`ere equation has been the subject of intensive studies in the past forty years, since the groundbreaking work by Yau (\cite{Y}) on the proof of the Calabi conjecture (\cite{Cal0}). Many interesting works were carried out by many people in the past several decades (see e.g. \cite{Au,CY,CKNS,MY,Kob,TY1,TY2,TY3,TY4,Ti1,Ti2,Gb,GPF}).

\medskip

Given all the above nice connections mentioned above, we believe that it will be very interesting to study the vector bundle version of the complex Hessian equation (\ref{Bk}). In this paper, we are going to study the uniqueness problem of the solution. Before doing that, one has to address the issue about the natural set for the solutions as the G\"arding cone for the standard Hessian equation for functions (see for example \cite{Wang}). However, the situation here is much more complicated than the scalar case.
In the following, the space of Hermitian metrics on $E$ will be denoted by $\emph{Herm}(E)$. For  any Hermitian metric $H\in \emph{Herm}(E)$, we denote
\begin{equation}\label{01}
S_{H}(E)=\left\{\eta \in \Omega^{0}(M, End(E))\,| \, \eta ^{\ast H}=\eta \right\},
\end{equation}
\begin{equation}\label{02}
S_{H, \epsilon}(E)=\left\{\eta \in S_{H}(E)\, | \ \sup_{M}|\eta |_{H}< \epsilon , \ \sup_{M}\left(|\partial_{H}\eta |_{H}+|\overline{\partial } \partial_{H}\eta |_{H}\right)< \epsilon \right\}
\end{equation}
and
\begin{equation}\label{03}
B_{H, \epsilon}(E)=\left\{K \in \emph{Herm}(E)\, | \ \log H^{-1}K \in  S_{H, \epsilon}(E)\right\}.
\end{equation}

\medskip

Now, we introduce a {\it positivity condition}, which is a natural generalization for the $\Gamma_k$ cone for scalar Hessian equation.

\begin{definition} Let $H$ be a Hermitian metric on the holomorphic vector  bundle $(E, \overline{\partial }_{E})$ over a compact K\"ahler manifold $(M, \omega )$. The curvature tensor $\sqrt{-1}F_{H}$ is said to be $\sigma_{k}$-$\omega$-positive on a point $p\in M$ if, for any non-zero $End(E)$-valued $(0, 1)$-form $\xi$, we have
\begin{equation}\label{p1}
\frac{\omega^{n-k}}{(n-k)!}\wedge  \sqrt{-1}\,\tr\left[\sum_{i=0}^{k-1}\left(\sqrt{-1}F_{H}\right)^{i}\wedge \xi^{\ast H} \wedge \left(\sqrt{-1}F_{H}\right)^{k-1-i}\wedge \xi \right]\bigg|_{p}>0.
\end{equation}
We denote the set of Hermitian metrics whose curvature are $\sigma_{k}$-$\omega$-positive on $M$ by $\mathcal{H}_{\omega , k}$.
\end{definition}

\medskip


It is a natural question whether the set $\mathcal{H}_{\omega , k}$ is indeed convex or not. That is, for any two metrics $K, H \in \mathcal{H}_{\omega , k}$, let $H(t)$ be the geodesic connecting $K$ and $H$, do we have $H(t)\in \mathcal{H}_{\omega , k}$? Indeed, if one can give a positive answer to the above question, by the second variation formula (\ref{v2}) for the generalized Donaldson's functional (\ref{Fk}), one can conclude that a solution $H\in \mathcal{H}_{\omega , k}$ of the bundle-valued complex  $k$-Hessian equation (\ref{Bk}) attains an absolute minimum of the generalized Donaldson's functional (\ref{Fk}), and then obtain the uniqueness of the solution. At present we can not prove the convexity of $\mathcal{H}_{\omega , k}$. Even if we can't give a positive answer to the above question, we can still prove local uniqueness of the solution of  bundle version complex  $k$-Hessian equation (\ref{Bk}) and the local minimum in the following theorem.

\begin{theorem}\label{thm1}
Let  $(E, \overline{\partial}_{E})$ be a holomorphic vector bundle over an $n$-dimensional compact K\"ahler manifold $(M, \omega )$, $H_{0}$ be a Hermitian metric on $E$. If $H$ is a solution of equation (\ref{Bk}) and   its curvature is $\sigma_{k}$-$\omega$-positive on $M$, then there exists a small number $\epsilon$ such that, for any $\tilde{H}\in B_{H, \epsilon}(E)$, we have
\begin{equation}\label{11}
  \mathcal{M} _{E, k }(H_{0}, H)\leq  \mathcal{M} _{E, k }(H_{0}, \tilde{H}).
 \end{equation}
Moreover, if the holomorphic bundle $(E, \overline{\partial}_{E})$ is simple and $K\in B_{H, \epsilon}(E) $ is another solution of (\ref{Bk}), then there  exists a  constant $a$ such that $K=e^aH$.
\end{theorem}

For the special case $k=2$, under the strongly $\sigma_{2}$-$\omega$-positive hypothesis (see the definition in \S 3), we obtain the following uniqueness theorem.

\begin{theorem}\label{thm2}
Let  $(E, \overline{\partial}_{E})$ be a simple holomorphic vector bundle over an $n$-dimensional compact K\"ahler manifold $(M, \omega )$, $H$ and $K$ be two Hermitian metrics on $E$ whose curvature are strongly $\sigma_{2}$-$\omega$-positive on the whole $M$. If both $H$ and $K$ are solutions of (\ref{Bk}) for $k=2$, then there must exist a  constant $a$ such that $K=e^aH$.
\end{theorem}

In the following, we want to discuss the $k$-stability of holomorphic vector bundles. First, let us recall the definition of the Chern classes (Chern character class) of coherent sheaf $\mathcal{F}$.
If the base manifold $M$ is a projective manifold, there exists a resolution
\begin{equation}
0\rightarrow E_{1} \rightarrow E_{2} \rightarrow \cdots \rightarrow E_{N}\rightarrow \mathcal{F}\rightarrow 0
\end{equation}
for each coherent sheaf $\mathcal{F}$ by locally free sheaves. The total Chern class $c(\mathcal{F})$ is defined by
\begin{equation}\label{CS}
c(\mathcal{F})=c(E_{N})c(E_{N-1})^{-1}\cdots c(E_{N-i})^{(-1)^{i}}\cdots c(E_{1})^{(-1)^{N-1}}.
\end{equation}
The definition is indeed independent of the chosen resolution as proved by Borel and Serre \cite{BS}. Over a non-projective compact complex manifold, the Chern classes of a coherent sheaf can be defined by the classes of Atiyah-Hirzenbruch (\cite{AH}, see \cite{Gr} for details). Note that on a non-projective compact complex manifold, a coherent sheaf does not necessarily have a global resolution via locally free coherent sheaves. Atiyah and Hirzebruch \cite{AH} circumvent this problem by using real-analytic functions. Using Grauert's results \cite{Gra}, they prove that, by taking tensor with the sheaf $\mathcal{A}_{M}$ of germs of real analytic functions on $M$, there is a resolution by real analytic vector bundles
\begin{equation}
0\rightarrow E_{1} \rightarrow E_{2} \rightarrow \cdots \rightarrow E_{N}\rightarrow \mathcal{F}\otimes_{\mathcal{O}_{M}}\mathcal{A}_{M}\rightarrow 0,
\end{equation}
where $E_{i}'s$ are real analytic complex vector bundles. Then we can define the total Chern class $c(\mathcal{F})$ as that in (\ref{CS}). The definition is also shown to be independent of the choice of resolution. By the way, we can obtain the definition of the Chern character classes $ch_{k}(\mathcal{F})$. We define the $k$-$\omega $-degree of $\mathcal{F}$ by
\begin{equation}
deg_{k, \omega }(\mathcal{F} )=\int_{M}ch_{k}(\mathcal{F})\wedge \frac{\omega^{n-k}}{(n-k)!}.
\end{equation}

\medskip

\begin{definition}
Let $(M, \omega )$ be a compact K\"ahler manifold,  and  $\mathcal{E}$ be a torsion-free coherent sheaf
over $M$.
  $\mathcal{E}$ is
called $k$-$\omega $-stable (semi-stable), if for every   coherent sub-sheaf
$\mathcal{F}\hookrightarrow \mathcal{E}$ of lower rank, it holds:
\begin{equation}
\mu_{k, \omega } (\mathcal{F})=\frac{deg_{k, \omega } (\mathcal{F})}{rank \mathcal{F}}< (\leq ) \mu_{k, \omega } (\mathcal{E})=\frac{deg_{k, \omega } (\mathcal{E})}{rank \mathcal{E}}.
\end{equation}
\end{definition}

\medskip

When $k=1$, $k$-$\omega $-stability is just the stability in the sense of Mumford-Takemoto, and is MA-stability (\cite{Pin}) when $k=n$. Following the classical Donaldson-Uhlenbeck-Yau theorem, it is natural to conjecture:  {\it Let $(E, \overline{\partial }_{E})$ be a simple holomorphic vector bundle over a compact K\"ahler manifold, and the set $\mathcal{H}_{\omega , k}$ is non empty.  Then $E$ admits a Hermitian metric $H$ satisfying the bundle valued complex $k$-Hessian equations (\ref{Bk}) if and only if it is $k$-$\omega $-stable}.

\medskip

This paper is organized as follows. In Section 2, we prove that the generalized Donaldson's functional is well defined, and give the first variation formula of this functional. In section 3, we give the proof of theorem \ref{thm1} and theorem \ref{thm2}.

\bigskip
\noindent
{\bf Acknowledgements:} The authors would like to thank Xiangwen Zhang for many helpful discussion.

\section{Generalized Donaldson's functionals}
Let $(M, \omega )$ be an $n$-dimensional compact K\"ahler manifold and $(E, \overline{\partial}_{E})$ a holomorphic vector bundle over $M$.
Given any two Hermitian metrics $H_{0}$ and $H$  on bundle $E$, let
\begin{eqnarray}
h=H_{0}^{-1}H\in \Gamma (End(E)).
\end{eqnarray}
It is easy to check that
\begin{equation}
h^{\ast H_{0}}=h^{\ast H}=h,
\end{equation}
\begin{equation}
D_{H}=D_{H_{0}}+h^{-1}\partial_{H_{0}}h
\end{equation}
and
\begin{equation}
F_{H}=F_{H_{0}}+\overline{\partial}(h^{-1}\partial_{H_{0}}h).
\end{equation}

\medskip

Moreover, we can prove the following lemma

\begin{lemma}\label{l1}
The functional $\mathcal{M} _{E, k }(H_{0}, H)$ is well defined, i.e. it is independent of the choice of path $H(s)$.
\end{lemma}

\begin{proof} Let $H(\tau, s)$ be a two-parameter family of metrics such that
\[
H(\tau , 0)=H_{0}\ \ \textit{\rm and}\ \ H(\tau , 1)=H,
\] where $\tau \in (-\epsilon , \epsilon )$, $s\in [0, 1]$. Then, $H(0 , s)$ and  $H(1 , s)$ are two paths connecting $H_{0}$ and $H$. Let
\[
h(\tau , s)=H_{0}^{-1}H(\tau , s).
\]
 It is easy to check that
\begin{equation}
\frac{\partial }{\partial \tau }F_{H(\tau , s)}=\overline{\partial }\partial_{H(\tau , s)}\left(h^{-1}\frac{\partial h}{\partial \tau }\right)
\end{equation}
and
\begin{equation}
\frac{\partial }{\partial s }F_{H(\tau , s)}=\overline{\partial }\partial_{H(\tau , s)}\left(h^{-1}\frac{\partial h}{\partial s }\right).
\end{equation}

\medskip

{\bf Claim: }
\begin{eqnarray}\label{claim1}
 &&\frac{\partial }{\partial \tau } \tr \left\{\left(\sqrt{-1}F_{H(\tau , s)}\right)^{k}h^{-1}\frac{\partial h}{\partial s}\right\}-\frac{\partial }{\partial s } \tr\left \{\left(\sqrt{-1}F_{H(\tau , s)}\right)^{k}h^{-1}\frac{\partial h}{\partial \tau}\right\}\\
&=& -\partial \tr \left\{ \sqrt{-1}\overline{\partial }\eta_{k} h^{-1}\frac{\partial h}{\partial s}\right\} -\overline{\partial } \tr\left\{ \sqrt{-1}\partial_{H} \phi_{k} h^{-1}\frac{\partial h}{\partial \tau}\right\}\nonumber
\end{eqnarray}
where
\begin{equation}\nonumber
\eta_{k}=\sum_{i=0}^{k-1}\left(\sqrt{-1}F_{H(\tau , s)}\right)^{i}h^{-1}\frac{\partial h}{\partial \tau }\wedge\left (\sqrt{-1}F_{H(\tau , s)}\right)^{k-1-i}
\end{equation}
and
\begin{equation}\nonumber
\phi_{k}=\sum_{i=0}^{k-1}\left(\sqrt{-1}F_{H(\tau , s)}\right)^{i}h^{-1}\frac{\partial h}{\partial s }\wedge \left(\sqrt{-1}F_{H(\tau , s)}\right)^{k-1-i}.
\end{equation}

\medskip

By the equality (\ref{claim1}), we have:
\begin{eqnarray}
& & \frac{d}{d \tau}\int_{0}^{1}\int_{M} \tr [(\sqrt{-1}F_{H(\tau , s)})^{k}H^{-1}(\tau , s)\frac{\partial H(\tau , s)}{\partial s}]\wedge \frac{\omega^{n-k}}{(n-k)!}ds\\\nonumber
&=&\int_{0}^{1}\int_{M}\frac{\partial }{\partial \tau } \{\tr [(\sqrt{-1}F_{H(\tau , s)})^{k}H^{-1}(\tau , s)\frac{\partial H(\tau , s)}{\partial s}]\}\wedge \frac{\omega^{n-k}}{(n-k)!}ds\\\nonumber
&=&\int_{0}^{1}\int_{M}\frac{\partial }{\partial s } \{\tr [(\sqrt{-1}F_{H(\tau , s)})^{k}H^{-1}(\tau , s)\frac{\partial H(\tau , s)}{\partial \tau }]\}\wedge \frac{\omega^{n-k}}{(n-k)!}ds\\\nonumber
&=&\int_{M} \{\tr [(\sqrt{-1}F_{H(\tau , s)})^{k}H^{-1}(\tau , s)\frac{\partial H(\tau , s)}{\partial \tau }]\}\wedge \frac{\omega^{n-k}}{(n-k)!}|_{s=0}^{1}\\\nonumber
&=& 0.
\end{eqnarray}
This shows that the functional $\mathcal{M} _{E, k }(H_{0}, H)$ is  independent of the choice of path $H(s)$. Therefore, we only need to prove the claim (\ref{claim1}).

\medskip

By applying the second Bianchi identity, we have
\begin{eqnarray}\label{c1}
& & \frac{\partial }{\partial \tau} \tr [(\sqrt{-1}F_{H(\tau , s)})^{k}H^{-1}(\tau , s)\frac{\partial H(\tau , s)}{\partial s}]\\\nonumber
&=& \sqrt{-1} \tr [\sum_{i=0}^{k-1}(\sqrt{-1}F_{H(\tau , s)})^{i}\wedge \overline{\partial }\partial_{H(\tau , s)}(h^{-1}\frac{\partial h}{\partial \tau })\wedge (\sqrt{-1}F_{H(\tau , s)})^{k-1-i}h^{-1}\frac{\partial h}{\partial s}]\\\nonumber
& &+ \tr [(\sqrt{-1}F_{H(\tau , s)})^{k}(-h^{-1}\frac{\partial h}{\partial \tau }h^{-1}\frac{\partial h}{\partial s}+h^{-1}\frac{\partial^{2} h}{\partial \tau \partial s })],\\\nonumber
&=&  \tr [\sqrt{-1} \overline{\partial }\partial_{H(\tau , s)}(\eta_{k}) h^{-1}\frac{\partial h}{\partial s}] + \tr [(\sqrt{-1}F_{H(\tau , s)})^{k}(-h^{-1}\frac{\partial h}{\partial \tau }h^{-1}\frac{\partial h}{\partial s}+h^{-1}\frac{\partial^{2} h}{\partial \tau \partial s })]
\end{eqnarray}
and similarly
\begin{eqnarray}\label{c2}
& & \frac{\partial }{\partial s} \tr [(\sqrt{-1}F_{H(\tau , s)})^{k}H^{-1}(\tau , s)\frac{\partial H(\tau , s)}{\partial \tau }]\\\nonumber
&=&  \tr [ \sqrt{-1}\overline{\partial }\partial_{H(\tau , s)}(\phi_{k}) h^{-1}\frac{\partial h}{\partial \tau}] + \tr [(\sqrt{-1}F_{H(\tau , s)})^{k}(-h^{-1}\frac{\partial h}{\partial s }h^{-1}\frac{\partial h}{\partial \tau }+h^{-1}\frac{\partial^{2} h}{\partial s \partial \tau })].
\end{eqnarray}
Using (\ref{c1}) and (\ref{c2}), we have
\begin{eqnarray}\label{c3}
&& \frac{\partial }{\partial \tau } \tr \{(\sqrt{-1}F_{H(\tau , s)})^{k}h^{-1}\frac{\partial h}{\partial s}\}-\frac{\partial }{\partial s } \tr \{(\sqrt{-1}F_{H(\tau , s)})^{k}h^{-1}\frac{\partial h}{\partial \tau}\}\\\nonumber
&=& -\tr [\sqrt{-1} \partial_{H(\tau , s)}\overline{\partial }(\eta_{k}) h^{-1}\frac{\partial h}{\partial s}]-\tr [ \sqrt{-1}\overline{\partial }\partial_{H(\tau , s)}(\phi_{k}) h^{-1}\frac{\partial h}{\partial \tau}] \\\nonumber
& & + \tr \{(\sqrt{-1}F_{H} \wedge \eta_{k} -\eta_{k}\wedge \sqrt{-1}F_{H}) h^{-1}\frac{\partial h}{\partial s}\}\\\nonumber
& & -\tr [(\sqrt{-1}F_{H(\tau , s)})^{k}(h^{-1}\frac{\partial h}{\partial \tau }h^{-1}\frac{\partial h}{\partial s })] +\tr [(\sqrt{-1}F_{H(\tau , s)})^{k}(h^{-1}\frac{\partial h}{\partial s }h^{-1}\frac{\partial h}{\partial \tau })].
\end{eqnarray}

By direct calculation, one can check that
\begin{eqnarray}\label{c4}
& &  \tr \{(\sqrt{-1}F_{H(\tau , s)} \wedge \eta_{k} -\eta_{k}\wedge \sqrt{-1}F_{H(\tau , s)}) h^{-1}\frac{\partial h}{\partial s}\}\\\nonumber
& & -\tr [(\sqrt{-1}F_{H(\tau , s)})^{k}(h^{-1}\frac{\partial h}{\partial \tau }h^{-1}\frac{\partial h}{\partial s })] +\tr [(\sqrt{-1}F_{H(\tau , s)})^{k}(h^{-1}\frac{\partial h}{\partial s }h^{-1}\frac{\partial h}{\partial \tau })]\\\nonumber
&=& 0,
\end{eqnarray}
and
\begin{eqnarray}\label{c5}
& &-\tr [\sqrt{-1} \partial_{H(\tau , s)}\overline{\partial }(\eta_{k}) h^{-1}\frac{\partial h}{\partial s}]-\tr [ \sqrt{-1}\overline{\partial }\partial_{H(\tau , s)}(\phi_{k}) h^{-1}\frac{\partial h}{\partial \tau}] \\\nonumber
&=&- \partial\tr [\sqrt{-1} \overline{\partial }(\eta_{k}) h^{-1}\frac{\partial h}{\partial s}]-\overline{\partial }\tr [ \sqrt{-1}\partial_{H(\tau , s)}(\phi_{k}) h^{-1}\frac{\partial h}{\partial \tau}] \\\nonumber
& & -\tr [\sqrt{-1} \overline{\partial }(\eta_{k})\wedge \partial_{H(\tau , s)} (h^{-1}\frac{\partial h}{\partial s})]-\tr [ \sqrt{-1}\partial_{H(\tau , s)}(\phi_{k})\wedge \overline{\partial }( h^{-1}\frac{\partial h}{\partial \tau})] \\\nonumber
&=&- \partial\tr [\sqrt{-1} \overline{\partial }(\eta_{k}) h^{-1}\frac{\partial h}{\partial s}]-\overline{\partial }\tr [ \sqrt{-1}\partial_{H(\tau , s)}(\phi_{k}) h^{-1}\frac{\partial h}{\partial \tau}] \\\nonumber
& & -\tr \{\sum_{i=0}^{k-1}(\sqrt{-1}F_{H})^{i}\wedge \overline{\partial }(h^{-1}\frac{\partial h}{\partial \tau })\wedge (\sqrt{-1}F_{H})^{k-1-i}\wedge \partial_{H }(h^{-1}\frac{\partial h}{\partial s })\}\\\nonumber
& & -\tr \{\sum_{j=0}^{k-1}(\sqrt{-1}F_{H})^{j}\wedge \partial_{H }(h^{-1}\frac{\partial h}{\partial s }) \wedge (\sqrt{-1}F_{H})^{k-1-j}\wedge \overline{\partial }(h^{-1}\frac{\partial h}{\partial \tau })\}\\\nonumber
&=&- \partial\tr [\sqrt{-1} \overline{\partial }(\eta_{k}) h^{-1}\frac{\partial h}{\partial s}]-\overline{\partial }\tr [ \sqrt{-1}\partial_{H(\tau , s)}(\phi_{k}) h^{-1}\frac{\partial h}{\partial \tau}] .
\end{eqnarray}
So (\ref{c3}), (\ref{c4}) and (\ref{c5}) imply the claim (\ref{claim1}).

\end{proof}

\medskip

In the following, we denote
\begin{equation}
\Psi_{k}(H)=\frac{(\sqrt{-1}F_{H})^{k}\wedge \frac{\omega^{n-k}}{(n-k)!}}{\frac{\omega^{n}}{n!}}.
\end{equation}
Let $H(t)$ be a family of metrics on $E$. We can consider two parameter family of metrics $H(t, s)$ such that $H(t, 0)=H_{0}$ and $H(t, 1)=H(t)$. By (\ref{claim1}), we have the following first variation formula:
\begin{eqnarray}\label{v1}
&&\frac{d}{d t}\mathcal{M} _{E, k }(H_{0}, H(t))\\\nonumber
&=&\int_{0}^{1}\int_{M}\frac{\partial }{\partial t } \{\tr [(\sqrt{-1}F_{H(t , s)})^{k}H^{-1}(t , s)\frac{\partial H(t , s)}{\partial s}]\}\wedge \frac{\omega^{n-k}}{(n-k)!}ds\\\nonumber
 &&-\frac{d}{d t}\int_{M}\lambda_{E, k}\log \det(h(t))\frac{\omega^{n}}{n!}\\\nonumber
&=& \int_{M}\tr \{[\Psi_{k}(H(t))-\lambda_{E, k}Id_{E}]H^{-1}(t)\frac{\partial H(t)}{\partial t}\} \frac{\omega^{n}}{n!}.
\end{eqnarray}
 Then we have the following lemma.

 \medskip

\begin{lemma}\label{l2}
 The Hermitian metric $H$ is a critical of the functional $\mathcal{M} _{E, k }$ if and only if it satisfies the equation
\begin{equation}\label{sk1}
\Psi_{k}(H)-\lambda_{E, k}Id_{E}=0.
\end{equation}
\end{lemma}

\medskip

By the definition of  $\mathcal{M} _{E, k }$ and choosing a special path, it is easy to obtain the following lemma.

\medskip

\begin{lemma}\label{l3} (1) For any metrics $H_{i}$, $i=0, 1, 2$, we have:
\begin{equation}
\mathcal{M} _{E, k } (H_{0}, H_{1})+ \mathcal{M} _{E, k } (H_{1}, H_{2})=\mathcal{M} _{E, k } (H_{0}, H_{2}).
\end{equation}

(2) $\mathcal{M} _{E, k } (H, aH)=0$ for any positive constant $a$.

\end{lemma}

\

\section{Some unique results }

In this section, we consider the uniqueness of the solution to the equation (\ref{Bk}). Let $K$ and $H$ be two Hermitian metrics on the bundle $E$. Set
\begin{equation}
s=\log H^{-1}K.
\end{equation}
The path $H(t)=He^{ts}$ will be called by the geodesic connecting $H$ and $K$. Along the geodesic, we have
\begin{equation}
h(t)^{-1}\frac{\partial h(t)}{\partial t}=s ,
\end{equation}
\begin{equation}
\frac{\partial }{\partial t}(h(t)^{-1}\partial_{H}h(t))=\partial_{H}s -sh(t)^{-1}\partial_{H}h(t) +h(t)^{-1}\partial_{H}h(t)s
\end{equation}
and
\begin{eqnarray}\nonumber
\frac{\partial }{\partial t}(F_{H(t)}-F_{H})&=&\frac{\partial }{\partial t}(\overline{\partial }_{E}(h(t)^{-1}\partial_{H}h(t)))=\overline{\partial }_{E}(\partial_{H(t)}s)\\\nonumber
&=&\overline{\partial }_{E}(\partial_{H}s)-(\overline{\partial}_{E}s)\wedge (h(t)^{-1}\partial_{H}h(t)) -(h(t)^{-1}\partial_{H}h(t))\wedge (\overline{\partial}_{E}s)\\\nonumber
 & &-s \overline{\partial }_{E}(h(t)^{-1}\partial_{H}h(t)) + \overline{\partial }_{E}(h(t)^{-1}\partial_{H}h(t)) s,
\end{eqnarray}
where $h(t)=H^{-1}H(t)=e^{ts}$.

From above identities, we obtain
\[
\frac{\partial }{\partial t}|h(t)^{-1}\partial_{H}h(t)|_{H}^{2}\leq 2|\partial_{H}s|_{H}|h(t)^{-1}\partial_{H}h(t)|_{H} +4|s|_{H}|h(t)^{-1}\partial_{H}h(t)|_{H} ,
\]
and
\begin{eqnarray}
& &\frac{\partial }{\partial t}|\overline{\partial }_{E}(h(t)^{-1}\partial_{H}h(t))|_{H}^{2}\\\nonumber
&\leq &2|\overline{\partial }_{E}(\partial_{H}s)|_{H}|\overline{\partial }_{E}(h(t)^{-1}\partial_{H}h(t))|_{H}+4|\overline{\partial}_{E}s|_{H}| h(t)^{-1}\partial_{H}h(t)|_{H}|\overline{\partial }_{E}(h(t)^{-1}\partial_{H}h(t))|_{H}\\ \nonumber
& &+4|s|_{H}|\overline{\partial }_{E}(h(t)^{-1}\partial_{H}h(t))|_{H}.
\end{eqnarray}
Similarly, we can also compute
\begin{eqnarray}
& &\frac{\partial }{\partial t}\left(\sqrt{|\overline{\partial }_{E}(h(t)^{-1}\partial_{H}h(t))|_{H}^{2}+\epsilon}+\sqrt{|h(t)^{-1}\partial_{H}h(t)|_{H}^{2}+\epsilon}\right)\\\nonumber
&\leq &|\overline{\partial }_{E}(\partial_{H}s)|_{H}+|\partial_{H}s|_{H}\\\nonumber
& &+4(|s|_{H}+|\overline{\partial}_{E}s|_{H})\left(\sqrt{|\overline{\partial }_{E}(h(t)^{-1}\partial_{H}h(t))|_{H}^{2}+\epsilon}+\sqrt{|h(t)^{-1}\partial_{H}h(t)|_{H}^{2}+\epsilon}\right)
\end{eqnarray}
and
\begin{eqnarray}\label{z1}
& &\frac{\partial }{\partial t}\left[e^{-4t(|s|_{H}+|\overline{\partial}_{E}s|_{H})}\left(\sqrt{|\overline{\partial }_{E}(h(t)^{-1}\partial_{H}h(t))|_{H}^{2}+\epsilon}+\sqrt{|h(t)^{-1}\partial_{H}h(t)|_{H}^{2}+\epsilon}\right)\right]\\\nonumber
&\leq &(|\overline{\partial }_{E}(\partial_{H}s)|_{H}+|\partial_{H}s|_{H})e^{-4t(|s|_{H}+|\overline{\partial}_{E}s|_{H})},
\end{eqnarray}
where $\epsilon$ is a positive number.
Integrating the equality (\ref{z1}) from $0$ to $t$, and letting $\epsilon \rightarrow 0$, we obtain
\begin{equation}\label{z2}
|\overline{\partial }_{E}(h(t)^{-1}\partial_{H}h(t))|_{H}+|h(t)^{-1}\partial_{H}h(t)|_{H}\leq (|\overline{\partial }_{E}(\partial_{H}s)|_{H}+|\partial_{H}s|_{H})\frac{e^{4t(|s|_{H}+|\overline{\partial}_{E}s|_{H})}-1}{4(|s|_{H}+|\overline{\partial}_{E}s|_{H})}.
\end{equation}

\

\begin{proof}[Proof of Theorem \ref{thm1}]$\\$

By (\ref{v1}), we have the following second variation formula for the generalized Donaldson's functional (\ref{Fk}):
\begin{eqnarray}\label{v2}
&&\frac{d^2}{d t^2}\mathcal{M} _{E, k }(H_{0}, H(t))\\\nonumber
&=& \int_{M}\tr \{[\Psi_{k}(H(t))-\lambda_{E, k}Id_{E}]\frac{\partial }{\partial t}[H^{-1}(t)\frac{\partial H(t)}{\partial t}]\} \frac{\omega^{n}}{n!}\\\nonumber
&& + \int_{M}\tr \{\frac{\partial }{\partial t}[\Psi_{k}(H(t))]H^{-1}(t)\frac{\partial H(t)}{\partial t}\} \frac{\omega^{n}}{n!}
\end{eqnarray}
and
\begin{eqnarray}\label{v21}
&&\frac{\partial }{\partial t}[\Psi_{k}(H(t))]\\\nonumber
&=&\frac{1}{\frac{\omega^{n}}{n!}}\{\frac{\omega^{n-k}}{(n-k)!}\wedge [\sum_{i=0}^{k-1}(\sqrt{-1}F_{H(t)})^{i}\wedge \sqrt{-1}\overline{\partial}\partial_{H(t)}(H^{-1}(t)\frac{\partial H(t)}{\partial t}) \wedge (\sqrt{-1}F_{H(t)})^{k-1-i}]\}.
\end{eqnarray}


Let $H$ and $K$ be two Hermitian metrics on $E$, and $H(t)=He^{ts}$  the  geodesic connecting $H$ and $K$, where $s=\log H^{-1}K$. By (\ref{v2}), (\ref{v21}) and the Bianchi identity, we get
\begin{eqnarray}\label{vs2}
&&\frac{d^2}{d t^2}\mathcal{M} _{E, k }(H_{0}, H(t))\\\nonumber
&=&  \int_{M}\tr \{\frac{\partial }{\partial t}[\Psi_{k}(H(t))]s\} \frac{\omega^{n}}{n!}\\\nonumber
&=& \int_{M}\frac{\omega^{n-k}}{(n-k)!}\wedge
 \tr[\sum_{i=0}^{k-1}(\sqrt{-1}F_{H(t)})^{i}\wedge \sqrt{-1}\partial_{H(t)}s \wedge (\sqrt{-1}F_{H(t)})^{k-1-i}\wedge \overline{\partial}_{E}s].
\end{eqnarray}

By the assumption that $H$ is a Hermitian metric whose curvature is $\sigma_{k}$-$\omega$-positive on $M$,  there exists a positive constant $\delta$ such that
\begin{equation}\label{p12}
\frac{\omega^{n-k}}{(n-k)!}\wedge  \sqrt{-1}\tr[\sum_{i=0}^{k-1}(\sqrt{-1}F_{H})^{i}\wedge \xi^{\ast H} \wedge (\sqrt{-1}F_{H})^{k-1-i}\wedge \xi ]\geq \delta |\xi |_{\omega , H}^{2}\frac{\omega^{n}}{n!}
\end{equation}
for any $\End(E)$-valued $(0, 1)$-form $\xi$. If $K\in B_{H, \epsilon}(E)$, by (\ref{z2}), it is straightforward to check that
\begin{eqnarray}\label{p13}
&&\frac{\omega^{n-k}}{(n-k)!}\wedge  \sqrt{-1}\tr[\sum_{i=0}^{k-1}(\sqrt{-1}F_{H(t)})^{i}\wedge \xi^{\ast H(t)} \wedge (\sqrt{-1}F_{H(t)})^{k-1-i}\wedge \xi ]\\\nonumber
&\geq &(\delta -C_{\epsilon} )|\xi |_{\omega , H}^{2}\frac{\omega^{n}}{n!},
\end{eqnarray}
where $C_{\epsilon}$ is a constant depending only on $\epsilon $ and $\sup_{M} |F_{H}|_{\omega , H}$, and $C_{\epsilon } \rightarrow 0$ as $\epsilon \rightarrow 0$. When $\epsilon $ is small enough, (\ref{p13}) implies that every $H(t)$ is a Hermitian metric whose curvature is also $\sigma_{k}$-$\omega$-positive on $M$. By (\ref{vs2}), we have
\begin{equation}\label{vs21}
\frac{d^2}{d t^2}\mathcal{M} _{E, k }(H_{0}, H(t))\geq  \frac{\delta }{2}\int_{M} |\overline{\partial}_{E}s|_{\omega , H}^{2} \frac{\omega^{n}}{n!}\geq 0
\end{equation}
for every $t\in [0, 1]$.

By the assumption that $H$ is a solution of the equation (\ref{Bk}), (\ref{v1}) implies that $$\frac{d}{d t}\mathcal{M} _{E, k }(H_{0}, H(t))|_{t=0}=0$$ and then $\frac{d}{d t}\mathcal{M} _{E, k }(H_{0}, H(t))\geq 0$ for every $t\in [0, 1]$. Hence we know that
\begin{equation}
\mathcal{M} _{E, k }(H_{0}, H)\leq \mathcal{M} _{E, k }(H_{0}, K)
\end{equation}
for every $K\in B_{H, \epsilon}(E)$ when $\epsilon $ is small enough, i.e. $H$ is a local minimum point of the generalized Donaldson's functional (\ref{Fk}).

Furthermore,  if $K\in B_{H, \epsilon}(E) $ is another solution of (\ref{Bk}), (\ref{vs21}) implies $\overline{\partial}_{E}s \equiv 0$, i.e. $s$ is a holomorphic section of $\End(E)$.  Under the assumption that  the holomorphic bundle $(E, \overline{\partial}_{E})$ is simple, we know that $s=aId_{E}$, and then  $K=e^aH$.

\end{proof}

\

Set $h=K^{-1}H$ and recall that
 \begin{equation}
 D_{H}-D_{K}=h^{-1}\partial_{K}h , \quad  F_{H}-F_{K}=\overline{\partial}_{E}(h^{-1}\partial_{K}h).
 \end{equation}
By direct calculation, we have
 \begin{eqnarray}\label{U1}
 &&(\sqrt{-1}F_{H})^{k}-(\sqrt{-1}F_{K})^{k}\\\nonumber
 &=&\sum_{i=1}^{k}(\sqrt{-1}F_{H})^{k-i+1}\wedge (\sqrt{-1}F_{K})^{i-1}-\sum_{i=1}^{k}(\sqrt{-1}F_{H})^{k-i}\wedge (\sqrt{-1}F_{K})^{i}\\\nonumber
 &=&\sum_{i=1}^{k}(\sqrt{-1}F_{H})^{k-i}\wedge (\sqrt{-1}F_{H}-\sqrt{-1}F_{K}) \wedge (\sqrt{-1}F_{K})^{i-1}\\\nonumber
 &=&\sum_{i=1}^{k}(\sqrt{-1}F_{H})^{k-i}\wedge (\sqrt{-1}\overline{\partial}_{E}(h^{-1}\partial_{K}h)) \wedge (\sqrt{-1}F_{K})^{i-1}.
 \end{eqnarray}

 \

In the following, we will focus on the case $k=2$.

\begin{definition} Let $H$ be a Hermitian metric on holomorphic vector  bundle $(E, \overline{\partial }_{E})$ over a compact K\"ahler manifold $(M, \omega )$. $H$ is said to be strongly $\sigma_{2}$-$\omega$-positive at a point $p\in M$ if its curvature $\sqrt{-1}F_{H}$ satisfies that, for every non-zero locally $End(E)$-valued $(0, 1)$-form $\xi$, there is
\begin{equation}\label{p1}
\frac{\omega^{n-2}}{(n-2)!}\wedge  \sqrt{-1}\tr[(\sqrt{-1}F_{H})\wedge \xi^{\ast H} \wedge  \xi ]|_{p}>0
\end{equation}
and
\begin{equation}\label{p2}
\frac{\omega^{n-2}}{(n-2)!}\wedge  \sqrt{-1}\tr[ \xi^{\ast H} \wedge (\sqrt{-1}F_{H})\wedge \xi ]|_{p}>0.
\end{equation}
\end{definition}

\medskip

We recall the definition of Nakano positivity and dual Nakano positivity.

\begin{definition}
A Hermitian metric $H$ on the holomorphic vector bundle $(E, \overline{\partial }_{E})$ is said to be Nakano-positive, if for any nonzero $n$-tuple  $\{u^{\alpha }\}_{\alpha =1 }^{n}$ of sections of $E$, it holds that
\begin{equation}
\langle F_{H}(\frac{\partial }{\partial z^{\alpha }}, \frac{\partial }{\partial \bar{z}^{\beta }})u^{\alpha }, u^{\beta }  \rangle_{H} >0.
\end{equation}
A metric $H$ is said to be dual  Nakano-positive, if for any nonzero $n$-tuple  $\{v^{\alpha }\}_{\alpha =1 }^{n}$ of sections of $E$, it holds that
\begin{equation}
\langle F_{H}(\frac{\partial }{\partial z^{\alpha }}, \frac{\partial }{\partial \bar{z}^{\beta }})v^{\beta }, v^{\alpha }  \rangle_{H} >0.
\end{equation}
\end{definition}

In \cite{LSY}, Liu, Sun and Yang found some Hermitian metrics that are both Nakano positive and dual Nakano positive. In the following, we will show that Nakano positivity implies (\ref{p1}) and dual Nakano positivity implies (\ref{p2}), i.e., we obtain the following lemma.

 \begin{lemma}
 Let  $(E, \overline{\partial}_{E})$ be a holomorphic vector bundle over an $n$-dimensional compact K\"ahler manifold $(M, \omega )$, $H$ be a Hermitian metric on $E$. If $H$ is Nakano-positive and dual Nakano-positive, then $H$ must be strongly $\sigma_{2}$-$\omega$-positive.
 \end{lemma}

\begin{proof}
Let $\xi$ be an $\End(E)$-valued $(0, 1)$-form. Choose suitable coordinates such that
\begin{equation*}
\omega=\sqrt{-1}dz^{\alpha}\wedge d\bar{z}^{\alpha}, \quad \xi=\xi_{\bar{\beta}}d\bar{z}^{\beta}, \quad F_{H}=F_{\alpha \bar{\beta }}dz^{\alpha }\wedge d\bar{z}^{\beta},
\end{equation*}
where $\xi_{\bar{\beta}}, F_{\alpha \bar{\beta }}\in \Gamma(\End(E))$.
After direct calculations, we find
\begin{equation}\label{N1}
\begin{split}
&\sqrt{-1}{\rm tr}((\sqrt{-1}F_H)\wedge \xi^{\ast H} \wedge \xi) \wedge \frac{\omega^{n-2}}{(n-2)!}\\
=&(\sqrt{-1})^2{\rm tr}(F_{\alpha\bar{\beta}}\cdot (\xi_{\bar{\gamma}})^{\ast H} \cdot \xi_{\bar{\delta}}) dz^{\alpha}\wedge d\bar{z}^{\beta}\wedge dz^{\gamma}\wedge d\bar{z}^{\delta} \wedge \frac{\omega^{n-2}}{(n-2)!}\\
=& (\sqrt{-1})^2 \sum_{\alpha\neq \beta}{\rm tr}(F_{\alpha\bar{\beta}} (\xi_{\bar{\beta}})^{\ast H} \xi_{\bar{\alpha}}) dz^{\alpha}\wedge d\bar{z}^{\beta}\wedge dz^{\beta}\wedge d\bar{z}^{\alpha} \wedge \frac{\omega^{n-2}}{(n-2)!}\\
& + (\sqrt{-1})^2 \sum_{\alpha\neq \beta}{\rm tr}(F_{\alpha\bar{\alpha}} (\xi_{\bar{\beta}})^{\ast H} \xi_{\bar{\beta}}) dz^{\alpha}\wedge d\bar{z}^{\alpha}\wedge dz^{\beta}\wedge d\bar{z}^{\beta} \wedge \frac{\omega^{n-2}}{(n-2)!}\\
=& \sum_{\alpha\neq \beta}{\rm tr}(F_{\alpha\bar{\alpha}} (\xi_{\bar{\beta}})^{\ast H} \xi_{\bar{\beta}}-F_{\alpha\bar{\beta}} (\xi_{\bar{\beta}})^{\ast H} \xi_{\bar{\alpha}})\frac{\omega^{n}}{n!}\\
=& \sum_{i=1}^r \sum_{\alpha\neq \beta} (\langle F_{\alpha\bar{\alpha}} (\xi_{\bar{\beta}})^{\ast H}(e_i),  (\xi_{\bar{\beta}})^{\ast H}(e_i) \rangle- \langle F_{\alpha\bar{\beta}} (\xi_{\bar{\beta}})^{\ast H}(e_i), (\xi_{\bar{\alpha}})^{\ast H}(e_i)\rangle)\frac{\omega^{n}}{n!}\\
=& \sum_{i=1}^r \sum_{\alpha< \beta} (\langle F_{\alpha\bar{\alpha}} (\xi_{\bar{\beta}})^{\ast H}(e_i),  (\xi_{\bar{\beta}})^{\ast H}(e_i) \rangle- \langle F_{\alpha\bar{\beta}} (\xi_{\bar{\beta}})^{\ast H}(e_i), (\xi_{\bar{\alpha}})^{\ast H}(e_i)\rangle\\
&+ \langle F_{\beta\bar{\beta}} (\xi_{\bar{\alpha}})^{\ast H}(e_i),  (\xi_{\bar{\alpha}})^{\ast H}(e_i) \rangle- \langle F_{\beta\bar{\alpha}} (\xi_{\bar{\alpha}})^{\ast H}(e_i), (\xi_{\bar{\beta}})^{\ast H}(e_i)\rangle)\frac{\omega^{n}}{n!},
\end{split}
\end{equation}
where $\{e_{i}\}_{i=1}^{r}$ is a unitary frame of $E$ with respect to the metric $H$.
   For every $\alpha< \beta$, set $u=(u^1, \cdots, u^n)$, $u^{\alpha}= (\xi_{\bar{\beta}})^{\ast H}(e_i)$, $u^{\beta}= -(\xi_{\bar{\alpha}})^{\ast H}(e_i)$, $u^{\gamma}=0$ for $\gamma\neq \alpha$ and $\gamma\neq \beta$. If the $n$-tuple of section $u$ is nonzero, by the assumption that $H$ is Nakano positive, we know
\begin{equation}\label{N2}
\begin{split}
0< &\sum_{\gamma, \delta}\langle F_{\gamma\bar{\delta}} u^{\gamma}, u^{\delta}\rangle=\sum_{\delta}(\langle F_{\alpha\bar{\delta}}(\xi_{\bar{\beta}})^{\ast H}(e_i), u^{\delta}\rangle- \langle F_{\beta\bar{\delta}}(\xi_{\bar{\alpha}})^{\ast H}(e_i), u^{\delta}\rangle)\\
=&\langle F_{\alpha\bar{\alpha}}(\xi_{\bar{\beta}})^{\ast H}(e_i), (\xi_{\bar{\beta}})^{\ast H}(e_i)\rangle- \langle F_{\alpha\bar{\beta}}(\xi_{\bar{\beta}})^{\ast H}(e_i), (\xi_{\bar{\alpha}})^{\ast H}(e_i)\rangle\\
&- \langle F_{\beta\bar{\alpha}}(\xi_{\bar{\alpha}})^{\ast H}(e_i), (\xi_{\bar{\beta}})^{\ast H}(e_i)\rangle + \langle F_{\beta\bar{\beta}}(\xi_{\bar{\alpha}})^{\ast H}(e_i), (\xi_{\bar{\alpha}})^{\ast H}(e_i)\rangle .
\end{split}
\end{equation}
From (\ref{N1}) and (\ref{N2}), we see that (\ref{p1}).

On the other hand,
\begin{equation}\label{N3}
\begin{split}
&\sqrt{-1}{\rm tr}(\xi^{\ast H}\wedge (\sqrt{-1}F_H)\wedge \xi) \wedge \frac{\omega^{n-2}}{(n-2)!}\\
=&(\sqrt{-1})^2{\rm tr}((\xi_{\bar{\alpha}})^{\ast H}\cdot F_{\gamma\bar{\delta}}\cdot \xi_{\bar{\beta}}) dz^{\alpha}\wedge dz^{\gamma}\wedge d\bar{z}^{\delta}\wedge d\bar{z}^{\beta}\wedge \frac{\omega^{n-2}}{(n-2)!}\\
=&\sum_{\alpha\neq \beta} (\sqrt{-1})^2 {\rm tr}((\xi_{\bar{\alpha}})^{\ast H} F_{\beta\bar{\alpha}}\xi_{\bar{\beta}})dz^{\alpha}\wedge dz^{\beta}\wedge d\bar{z}^{\alpha}\wedge d\bar{z}^{\beta}\wedge \frac{\omega^{n-2}}{(n-2)!}\\
&+\sum_{\alpha\neq \beta}  (\sqrt{-1})^2 {\rm tr}((\xi_{\bar{\alpha}})^{\ast H} F_{\beta\bar{\beta}}\xi_{\bar{\alpha}})dz^{\alpha}\wedge dz^{\beta}\wedge d\bar{z}^{\beta} \wedge d\bar{z}^{\alpha}\wedge \frac{\omega^{n-2}}{(n-2)!}\\
=& \sum_{\alpha\neq \beta} {\rm tr}((\xi_{\bar{\alpha}})^{\ast H} F_{\beta\bar{\beta}}\xi_{\bar{\alpha}}-(\xi_{\bar{\alpha}})^{\ast H} F_{\beta\bar{\alpha}}\xi_{\bar{\beta}}) \frac{\omega^{n}}{n!}\\
=& \sum_{i=1}^r \sum_{\alpha\neq \beta} (\langle F_{\beta\bar{\beta}} \xi_{\bar{\alpha}}(e_i), \xi_{\bar{\alpha}}(e_i)\rangle- \langle F_{\beta\bar{\alpha}}\xi_{\bar{\beta}}(e_i),  \xi_{\bar{\alpha}}(e_i)\rangle)\frac{\omega^{n}}{n!}\\
=& \sum_{i=1}^r \sum_{\alpha< \beta} (\langle F_{\beta\bar{\beta}} \xi_{\bar{\alpha}}(e_i), \xi_{\bar{\alpha}}(e_i)\rangle- \langle F_{\beta\bar{\alpha}}\xi_{\bar{\beta}}(e_i),  \xi_{\bar{\alpha}}(e_i)\rangle\\
&+ \langle F_{\alpha\bar{\alpha}} \xi_{\bar{\beta}}(e_i), \xi_{\bar{\beta}}(e_i)\rangle- \langle F_{\alpha\bar{\beta}}\xi_{\bar{\alpha}}(e_i),  \xi_{\bar{\beta}}(e_i)\rangle)\frac{\omega^{n}}{n!}.
\end{split}
\end{equation}
 For any $\alpha< \beta$, set $v=(v^1, \cdots, v^n)$, $v^{\alpha}= \xi_{\bar{\beta}}(e_i)$, $v^{\beta}= -\xi_{\bar{\alpha}}(e_i)$, $u^{\gamma}=0$ for $\gamma\neq \alpha$ and $\gamma\neq \beta$. If $v$ is nonzero, the dual Nakano positivity implies
 \begin{equation}\label{N4}
\begin{split}
0< &\sum_{\gamma, \delta}\langle F_{\gamma\bar{\delta}} v^{\delta}, v^{\gamma}\rangle=\sum_{\delta}(\langle F_{\alpha\bar{\delta}}v^{\delta}, \xi_{\bar{\beta}}(e_i)\rangle- \langle F_{\beta\bar{\delta}}v^{\delta}, \xi_{\bar{\alpha}}(e_i)\rangle)\\
=&\langle F_{\alpha\bar{\alpha}}\xi_{\bar{\beta}}(e_i), \xi_{\bar{\beta}}(e_i)\rangle- \langle F_{\alpha\bar{\beta}}\xi_{\bar{\alpha}}(e_i), \xi_{\bar{\beta}}(e_i)\rangle\\
&- \langle F_{\beta\bar{\alpha}}\xi_{\bar{\beta}}(e_i), \xi_{\bar{\alpha}}(e_i)\rangle + \langle F_{\beta\bar{\beta}}\xi_{\bar{\alpha}}(e_i), \xi_{\bar{\alpha}}(e_i)\rangle,
\end{split}
\end{equation}
then (\ref{N3}) and (\ref{N4}) mean (\ref{p2}).

\end{proof}

 Now, we consider the uniqueness of the solution of the following bundle valued complex $2$-Hessian equation
 \begin{equation}\label{s2}
 \frac{(\sqrt{-1}F_{H})^{2}\wedge \frac{\omega^{n-2}}{(n-2)!}}{\frac{\omega^{n}}{n!}}=\lambda_{E, 2}Id_{E}.
 \end{equation}

\begin{proof}[Proof of Theorem \ref{thm2}]By (\ref{U1}), we have
 \begin{equation}
 \begin{array}{ll}
 &(\sqrt{-1}F_{H})^{2}-(\sqrt{-1}F_{K})^{2}\\
 &=\sqrt{-1}F_{H}\wedge (\sqrt{-1}F_{H}-\sqrt{-1}F_{K})+(\sqrt{-1}F_{H}-\sqrt{-1}F_{K})\wedge \sqrt{-1}F_{K}\\
 &=\sqrt{-1}F_{H}\wedge \sqrt{-1}(\overline{\partial}_{E}(\partial_{H}h h^{-1}))+\sqrt{-1}(\overline{\partial}_{E}(h^{-1}\partial_{K}h))\wedge \sqrt{-1}F_{K}.\\
 \end{array}
 \end{equation}
Using the  strongly $\sigma_{2}$-$\omega$-positivity hypothesis, the Bianchi identity and Stokes formula, we derive
\begin{eqnarray}
\nonumber 0 &=&\int_{M}\tr \{[(\sqrt{-1}F_{H})^{2}-(\sqrt{-1}F_{K})^{2}]h\}\wedge \omega^{n-2}\\\nonumber
  &=&\int_{M}\tr \{[\sqrt{-1}F_{H}\wedge \sqrt{-1}(\overline{\partial}_{E}(\partial_{H}h h^{-1}))+\sqrt{-1}(\overline{\partial}_{E}(h^{-1}\partial_{K}h))\wedge \sqrt{-1}F_{K}]h\}\wedge \omega^{n-2}\\\nonumber
  &=&\int_{M}\tr \{\sqrt{-1}F_{H}\wedge (\sqrt{-1}\partial_{H}h h^{-\frac{1}{2}})\wedge h^{-\frac{1}{2}}\overline{\partial}_{E}h\}\wedge \omega^{n-2}\\\nonumber
  &&+ \int_{M}\tr \{\sqrt{-1}(h^{-\frac{1}{2}}\partial_{K}h)\wedge \sqrt{-1}F_{K}\wedge \overline{\partial}_{E}h h^{-\frac{1}{2}}\}\wedge \omega^{n-2}\\\nonumber
  &\geq& \tilde{c}\int_{M} \{|h^{-\frac{1}{2}}\overline{\partial}_{E}h|_{H}^{2}+|\overline{\partial}_{E}h h^{-\frac{1}{2}}|_{K}^{2}\}\omega^{n}
 \end{eqnarray}
 for some positive constant $\tilde{c}$, and
 then
 \begin{equation}
 \overline{\partial}_{E}h=0.
 \end{equation}
The simplicity of $(E , \overline{\partial }_{E})$ gives us  $h=e^{-a}\,{\rm Id}_{E}$, and then $K=e^{a}H$.

\end{proof}

\

\bibliographystyle{plain}

\end{document}